\documentclass{article}

\usepackage{amsmath,amssymb,latexsym,graphics,epsfig}
\usepackage{color}
\usepackage{amsthm}
\usepackage{graphicx,url}
\usepackage{hyperref}

\newtheorem{thm}{Theorem}[section]
\newtheorem{prp}[thm]{Proposition}
\newtheorem{lem}[thm]{Lemma}

\def\t{{\!\top\!}}
\def\0{{\bf 0}}
\def\1{{\bf 1}}
\def\F2{\mathbb{F}_{2}}

\begin{document}

\title
{\bf Switched symplectic graphs \\ and their 2-ranks}
\author
{Aida Abiad and Willem H. Haemers
\\
{\it\small Department of Econometrics and Operations Research,}
\\
{\it\small Tilburg University, Tilburg, The Netherlands}\\
{\small {\tt
A.AbiadMonge@uvt.nl, Haemers@uvt.nl}} \\
}
\date{}
\maketitle

\begin{abstract}
\noindent
We apply Godsil-McKay switching to the symplectic graphs over $\mathbb{F}_2$ with at least 63 vertices and prove that
the 2-rank of (the adjacency matrix of) the graph increases after switching. This shows that the switched graph
is a new strongly regular graph with parameters $(2^{2\nu}\!-1, 2^{2\nu-1}, 2^{2\nu-2},2^{2\nu-2})$ and
2-rank $2\nu+2$ when $\nu\geq 3$.
For the symplectic graph on $63$ vertices we investigate repeated switching by computer and find many new
strongly regular graphs with the above parameters for $\nu=3$ with various 2-ranks.
Using these results and a recursive construction method for the symplectic graph from Hadamard matrices,
we obtain several graphs with the above parameters, but different 2-ranks for every $\nu\geq 3$.
\\[7pt]
\end{abstract}

\section{Introduction}
A Godsil-McKay switching set of a graph is a special type of subset of the vertices, that makes it possible
to switch some of the edges such that the spectrum of the adjacency matrix doesn't change.
It is well-known that if a graph $G'$ has the same spectrum as a strongly regular graph $G$,
then $G'$ is also strongly regular with the same parameters as $G$ (see for example~\cite{BH}).
Therefore Godsil-McKay switching provides a tool to construct new strongly regular graphs from known ones.
However, there is no guarantee that the switched graph is non-isomorphic with the original graph.
In this paper we use the 2-rank of the adjacency matrix to prove non-isomorphism after switching.

For $\nu\geq 2$, the symplectic graph over $\F2$, denoted by $Sp(2\nu,2)$ is a strongly regular with
parameters $(2^{2\nu}\!-1, 2^{2\nu-1}, 2^{2\nu-2},2^{2\nu-2})$.
The 2-rank of the adjacency matrix of $Sp(2\nu,2)$ equals $2\nu$, which is the smallest possible value.
The symplectic graph is characterized by Peeters~\cite{P1995} as follows.
\begin{thm}\label{Peeters}
The symplectic graph $Sp(2\nu,2)$ is uniquely determined by its parameters and its $2$-rank.
\end{thm}
When $\nu=2$ we have the complete graph $K_3$, and $Sp(4,2)$ is a strongly regular graph with parameters
$(15,8,4,4)$, which is known to be determined by its parameters.
For $\nu\geq 3$ we find Godsil-McKay switching sets in $Sp(2\nu,2)$ and prove that the 2-rank increases
after switching, which implies that the switched graph is non-isomorphic with the original graph.

It turns out that for $\nu\geq 3$ the symplectic graph has many switching sets that remain switching sets
after switching.
Therefore it is interesting to find out what happens after several switchings.
We investigated this by computer for $Sp(6,2)$ and found 1826 new strongly regular graphs with parameters
$(63,32,16,16)$.
For the 2-rank of these new graphs we found six different values.

The symplectic graphs will be defined below, and in Section~\ref{Hadamard} we give an alternative
description by use of a well-known recursive construction of Hadamard matrices.
We settle the behavior of the 2-ranks of this recursive construction, and apply it to the
strongly regular graphs with various 2-ranks found by computer.
As a result we find that for every $\nu\geq 3$ there exist strongly regular graphs with the same parameters
as $Sp(2\nu,2)$ for a number of distinct values for the 2-rank.
Moreover, this number of different 2-ranks is nondecreasing and goes to infinity when $\nu\rightarrow\infty$.

\subsection{The symplectic graphs over $\F2$}
Let $\F2^{2\nu}$ be the $2\nu$-dimensional vector space over $\F2$, and let $K=I_{\nu}\otimes (J_2-I_2)$, where
$I_{\nu}$ is the identity matrix of order $\nu$, and $J$ denotes the all-ones matrix of order $2$.
The \emph{symplectic graph} $Sp(2\nu,2)$ over $\F2$ is the graph whose vertices are the nonzero vectors
of $\F2^{2\nu}$, where two vertices $x$ and $y$ are adjacent whenever $x^{\t}Ky = 1$.
Equivalently, $x=[x_1\ \ldots\ x_{2\nu}]^\t$ and $y=[y_1\ \ldots\ y_{2\nu}]^\t$ are adjacent if
\[ (x_{1}y_2 + x_2 y_{1})+(x_{3}y_4+x_{4}y_{3})+\cdots+(x_{2\nu-1}y_{2\nu}+x_{2\nu}y_{2\nu-1})=1.
\]
For $\nu\geq 2$, it is known (see for example \cite{P1995}) that the symplectic graph $Sp(2\nu,2)$ is a
strongly regular graph with parameters
\begin{equation*}
\left(2^{2\nu}-1,\ 2^{2\nu-1},\ 2^{2\nu-2},\ 2^{2\nu-2}\right).
\end{equation*}

\subsection{Godsil-McKay switching and its 2-rank behavior}
Godsil and McKay~\cite{GM1982} proved the following result.
\begin{thm}\label{GM}
Let $G$ be a graph and let $S$ be a subset of the vertex set of $G$ which induces a regular subgraph.
Assume that each vertex outside $S$ is adjacent to $|S|$, $\frac{1}{2}|S|$ or $0$ vertices of $S$.
Make a new graph $G'$ from $G$ as follows.
For each vertex $v$ outside $S$ with $\frac{1}{2}|S|$ neighbors in $S$,
delete the $\frac{1}{2}|S|$ edges between $v$ and $S$, and join $v$ instead to the $\frac{1}{2}|S|$
other vertices in $S$.
Then $G$ and $G'$ have the same spectrum.
\end{thm}
The operation that changes $G$ into $G'$ is called \emph{Godsil-McKay switching}.
The subset $S$ of the vertex set of $G$ will be called a \emph{(Godsil-McKay) switching set}.
Note that any vertex subset of $G$ of size 2 satisfies the required conditions,
but in this case the switched graph $G'$ is isomorphic with $G$.
Therefore we assume that a switching set has at least four vertices.

Let $A$ and $A'$ be the adjacency matrices of $G$ and $G'$, respectively,
and assume that the first $|S|$ rows (and columns) of $A$ and $A'$
correspond to the switching set $S$ and the last $h$ rows correspond to
the vertices outside $S$ with exactly $\frac{1}{2}|S|$ neighbors in $S$.
Then
\[
A'=A+K~(\mbox{mod }2),~~\mbox{where}~~
K=\left[\begin{array}{ccc}
O&O&J\\O&O&O\\\!J^\top\!&O&O
\end{array}\right] ,
\]
and $J$ is the $|S|\times h$ all-ones matrix.
Since 2-rank$(K)=2$, the 2-ranks of $A$ and $A'$ differ by at most $2$.
It is well-known that the 2-rank of any adjacency matrix is even (see~\cite{BE1992}),
thus we have the following result.
\begin{prp}\label{GMrank}
Suppose $2$-rank$(A)=r$, then $r$ is even and $2$-rank$(A')=r-2$, $r$, or $r+2$.
\end{prp}

\section{Switched symplectic Graphs}

For $\nu\geq 3$, we define the following vectors in $\F2^{2\nu}$:
\begin{gather*}
v_1=\left[\begin{array}{c} 1 \\ 0 \\ 1 \\ 0 \\ 1 \\ 0 \\ z\end{array}\right],
v_2=\left[\begin{array}{c} 1 \\ 0 \\ 0 \\ 1 \\ 0 \\ 1 \\ z\end{array}\right],
v_3=\left[\begin{array}{c} 0 \\ 1 \\ 1 \\ 0 \\ 0 \\ 1 \\ z\end{array}\right],
v_4=\left[\begin{array}{c} 0 \\ 1 \\ 0 \\ 1 \\ 1 \\ 0 \\ z\end{array}\right],
\end{gather*}
where $z$ is an arbitrary vector in $\F2^{2\nu-6}$.
\begin{prp}\label{prp:ss}
The set $S=\{v_1,v_2,v_3,v_4\}$ is a Godsil-McKay switching set of $Sp(2\nu,2)$ for $\nu \geq 3$.
\end{prp}
\begin{proof}
Any two vertices from $S$ are nonadjacent, so the subgraph of $Sp(2\nu,2)$
induced by $S$ is a coclique, and therefore regular.
Consider an arbitrary vertex $x\not\in S$.
Then
\[
x^{\t}Kv_1+x^{\t}Kv_2+x^{\t}Kv_3+x^{\t}Kv_4 = x^{\t}K(v_1+v_2+v_3+v_4) = x^\t K \0 = 0.
\]
This implies that the number of edges between $x$ and $S$ is even, and therefore $S$ is a switching set.
\end{proof}

Let $G'$ be the graph obtained from $G=Sp(2\nu,2)$ by switching
with respect to $S$. We shall now prove that $G$ and $G'$
are non-isomorphic.
\begin{thm}\label{switch}
For $\nu \geq 3$, the graph $G'$ obtained from $Sp(2\nu,2)$ by switching with respect to the switching set $S$
given above, is strongly regular with the same parameters as $Sp(2\nu,2)$, but with $2$-rank equal to $2\nu+2$.
\end{thm}
\begin{proof}
Let $A$ be the adjacency matrix of $G=Sp(2\nu,2)$,
and assume that the first four rows and columns correspond to $S$.
Then 2-rank$(A)=2\nu$ and $A$ has $2^{2\nu}-1$ rows.
This implies that, over $\F2$, every possible nonzero linear combination of a basis of the row space of $A$
is a row of $A$.
Therefore the sum (mod~2) of any two rows of $A$ is again a row of $A$.
Let $r_5$ and $r_6$ be rows of $A$ corresponding to the vertices
$v_5=[1 0 0 0 0 0 {z^{\!\top}}]^\top\!$ and $v_6=[0 0 1 0 0 0 {z^{\!\top}}]^\top\!$, respectively.
Then $r_1$ starts with $0011$ and $r_2$ starts with $0101$.
It follows that $r_7=r_5+r_6$ is also a row of $A$ starting with $0110$.
After switching only the first four entries of $r_5$, $r_6$ and $r_7$ change and become
$1100$, $1010$ and $1001$, respectively.
Let $r'_i$ denote the switched version of $r_i$ ($i=5, 6$ or 7).
Then $r'_5+r'_6+r'_7= 11110\ldots 0$, and this is in the row space of the switched matrix $A'$, but it is not a row of $A'$.
So $G'$ is not isomorphic to $G$, and by Theorem~\ref{Peeters} and Proposition~\ref{GMrank} the
2-rank of $A'$ equals $2\nu+2$.
\end{proof}

The switching set $S$ given above, is not the only one.
There are many more (indeed, for any three independent vectors $v_1$, $v_2$ and $v_3\in\F2^{2\nu}$,
the set $\{v_1,\ v_2,\ v_3,\ v_1+v_2+v_3\}$ is a Godsil-MvKay switching set whenever the four vertices induce a
regular subgraph.) Therefore we can apply switching several times. However it is not true in general that a second switching increases the 2-rank again,
and it looks difficult to make a general statement like in the above theorem.
Instead we investigated the repeated switching by computer for the case $\nu=3$.

It is also worthwhile to mention that in \cite{HPR1999} upper bounds for the $2$-rank of strongly regular graphs in terms of the eigenvalues are given.
For graphs with the same parameters of $Sp(2\nu,2)$, the 2-rank of its adjacency matrix $A$ is bounded from above by $2^{2\nu-1}-2^{\nu-1}$. But in fact, it can be improved, since the spectrum implies that the matrix
$$E=(2^{2\nu}-1)\left(A+(2^{\nu-1})I\right)-(2^{2\nu-1}+2^{\nu-1})J$$
has real rank equal to $2^{2\nu-1}-2^{\nu-1}-1$, and therefore the 2-rank of $E$ is at most $2^{2\nu-1}-2^{\nu-1}-1$. Since $A\equiv E$ (mod $2$) and the 2-rank of $A$ must be even, it follows that the 2-rank is upper bounded by $2^{2\nu-1}-2^{\nu-1}-2$.

\section{Repeated switching in $Sp(6,2)$}\label{computer}

In this section we show that Godsil-McKay switching generates a significant number of non-isomorphic graphs with
the same parameters as the symplectic graph $Sp(6,2)$.
By computer we search for all switching sets of size $4$ in $Sp(6,2)$.
We switch and compute the 2-rank.
With the firstly encountered graph for which the 2-rank has increased, we repeat the procedure.
We stop if the 2-rank cannot be increased.
By this procedure we obtained 1827 non-isomorphic graphs with the parameters of $Sp(6,2)$.
The possible 2-ranks are: 6, 8, 10, 12, 14, 16 and 18.
No doubt we would have obtained many more non-isomorphic graphs with these parameters if we would have
continued the search for other graphs for which the $2$-rank has increased after switching.
But the isomorphism tests are very time consuming, and since we are mainly interested in the $2$-ranks,
we choose not to do so.
We did, however, continue with some other graphs without worrying about isomorphism in the hope to
find examples with a $2$-rank of $20$ (or more), without success.

We will not display all newly obtained strongly regular graphs, instead we just give the sequence of
switching sets that increases the $2$-rank in each step (vertices are represented as row vectors):
\[
\begin{array}{c}
\{(100000),(010000),(101000),(011000)\},\\
\{(100000),(010000),(100100),(010100)\},\\
\{(100000),(010000),(100010),(010010)\},\\
\{(100000),(010000),(100001),(010001)\},\\
\{(110000),(001000),(000010),(111010)\},\\
\{(110000),(001000),(000001),(111001)\}.
\end{array}
\]
The number of cospectral graphs obtained in each of the six above iterations is: $4275$ with 2-rank $8$
($161$ are non-isomorphic), $2238$ with 2-rank $10$ ($195$ are non-isomorphic), $1242$ with 2-rank $12$
($301$ are non-isomorphic), $818$ with 2-rank $14$ ($489$ are non-isomorphic), $508$ with 2-rank $16$
($508$ are non-isomorphic) and $172$ with 2-rank $18$ ($172$ are non-isomorphic).

Thus we see that there is still a gap between the constructed cases and the theoretic upper bound
for the 2-rank mentioned in Section~\ref{ssg}, which for $Sp(6,2)$ is 26.

\section{Hadamard matrices and $2$-ranks}\label{Hadamard}

We recall some results of Hadamard matrices.
A square $(+1,-1)$-matrix $H$ of order $n$ is a \emph{Hadamard matrix} (or $H$-matrix) whenever $HH^{\t}=nI$.
For example

\begin{equation*}
H=\left[\begin{array}{cccc}
1 & 1 & 1 & 1 \\
1 & 1 & - & - \\
1 & - & 1 & - \\
1 & - & - & 1
\end{array} \right]
\end{equation*}
\noindent is a Hadamard matrix of order $4$ (we write $-$ instead of $-1$).
If a row or a column of a Hadamard matrix is multiplied by $-1$, it remains a Hadamard matrix.
We can multiply rows and columns of any Hadamard matrix by $-1$ such that the first row and column
consist of all ones.
Such a Hadamard matrix is called \emph{normalized}.
A Hadamard matrix $H$ is said to be \emph{graphical} if $H$ is symmetric and it has constant diagonal.
Note that if $H$ is a graphical Hadamard matrix of order $n$ with $\delta$ on the diagonal,
then $A=\frac{1}{2}(J-\delta H)$ is the adjacency matrix of a graph on $n$ vertices.
If $H$ is normalized, the obtained graph has an isolated vertex, and it is well-known that for $n>4$
the graph on the remaining $n-1$ vertices is strongly regular with parameters
$(n-1 , n/2, n/4, n/4)$.
And conversely, any strongly regular graph with the above parameters comes from a graphical Hadamard matrix.
For example, the above Hadamard matrix $H$ is graphical and normalized.
The corresponding graph is the smallest symplectic graph $Sp(2,2)=K_3$ extended with an isolated vertex.
It is well known that if $H_1$ and $H_2$ are Hadamard matrices,
then so is the Kronecker product $H_1\otimes H_2$.
Moreover, if $H_1$ and $H_2$ are normalized, then so is $H_1\otimes H_2$, and if $H_1$ and $H_2$ are graphical,
then so is $H_1\otimes H_2$.
For a Hadamard matrix $H$, we define $A_H = \frac{1}{2}(J-H)$ and $\rho (H)=\mbox{2-rank}(A_{H})$.
\begin{lem}\label{rank}
Let $H_1$ and $H_2$ be two Hadamard matrices, then $\rho (H_1\otimes H_2)\leq \rho (H_1)+\rho (H_2)$,
with equality if $H_1$ and $H_2$ are normalized.
\end{lem}
\begin{proof}
It is easily seen that
\[
A_{H_1\otimes H_2} = (J\otimes A_{H_1}) + (A_{H_2}\otimes J) \quad (\text{mod }2).
\]
For any integer matrix $A$ we have
$\mbox{2-rank}(J\otimes A)=\mbox{2-rank}(A\otimes J)=\mbox{2-rank}(A)$.
Therefore $\rho(H_1\otimes H_2)\leq \rho(H_1)+\rho(H_2)$.

To prove the second statement, we define $V_i$ to be a matrix consisting of $\rho(H_i)$ independent columns
of $A_{H_i}$ for $i=1,2$ (so the columns of $V_1$ and $V_2$ form a basis for the column space of $A_{H_1}$
and $A_{H_2}$, respectively).
Suppose $H_1$ and $H_2$ are normalized.
Then $A_{H_1\otimes H_2}$ contains the columns of $\1\otimes V_1$ and $V_2\otimes \1$.
These $\rho(H_1)+\rho(H_2)$ columns are independent (indeed, the first rows of $V_1$ and $V_2$ are all-zero
rows and therefore the only vector in the intersection of the column space of $\1\otimes V_1$ and the column
space of $V_2\otimes \1$ is the zero vector), and hence $\rho(H_1\otimes H_2)=\rho(H_1)+\rho(H_2)$.
\end{proof}
With the Hadamard matrix $H$ of order $4$, given above we define
\[
H^{\otimes\nu}=H\otimes H\otimes \cdots \otimes H \ (\nu\ \mbox{times}).
\]
Then clearly $H^{\otimes\nu}$ is a normalized graphical Hadamard matrix of order $4^{\nu}$,
and Lemma~\ref{rank} implies that 2-rank$(A_{H^{\otimes\nu}})=\rho(H^{\otimes\nu})=2\nu$.
Therefore, by Theorem~\ref{Peeters} the strongly regular graph associated with $H^{\otimes\nu}$ is the
symplectic graph $Sp(2\nu,2)$.

In the definition of $H^{\otimes\nu}$ we can replace any triple product $H\otimes H\otimes H$
by any other regular graphical Hadamard matrix of order $64$.
By choosing Hadamard matrices coming from the strongly regular graphs with various $2$-ranks
found by computer in Section~\ref{computer}, we obtain normalized graphical Hadamard matrices of order $4^\nu$,
and the $2$-rank of the associated strongly regular can take all even values between $2\nu$ and
$2\nu+12\lfloor \nu/3 \rfloor$.
Thus we find:
\begin{thm}\label{z}
For any even $r\in[2\nu,\ 2\nu+12\lfloor \nu/3 \rfloor]$ there exists a strongly regular graph with parameters
$(2^{2\nu}-1,2^{\nu-1},2^{\nu-2},2^{\nu-2})$ and $2$-rank $r$.
\end{thm}
Another application of Lemma~\ref{rank} is the following.
There exist strongly regular graphs with parameters $(35,18,9,9)$ for the $2$-ranks $6$, $8$, $10$, $12$
and $14$; see~\cite{HPR1999}.
Let $H^*$ be the associated normalized graphical Hadamard matrix of order 36, and let $H$ be as before,
then $H\otimes H^*$ is associated with a strongly regular graph with parameters $(143,72,36,36)$.
By Lemma~\ref{rank} we find that such strongly regular graphs exist for every even $2$-rank
between $8$ and $16$.
\section{Remarks}
A different construction of graphs with the same parameters as $Sp(2\nu,2)$ was given by
Munemasa and Vanhove~\cite{MV2014}.
It would be interesting to know the $2$-rank of their construction.
It is claimed in \cite{MV2014} that the construction admits a cyclic difference set,
and using Corollary 3.7 from \cite{AHPX2002}, it follows that the $2$-rank is a multiple of
$2\nu$, and therefore at least $4\nu$.
So we can conclude that their graphs are not isomorphic to the ones obtained in Theorem~\ref{switch}.
\\
A graph associated with a normalized graphical Hadamard matrix, is a so-called $(v,k,\lambda)$ graph,
which means that the adjacency matrix can be interpreted as the incidence matrix of a symmetric
2-$(v,k,\lambda)$ design.
It is possible that non-isomorphic graphs lead to isomorphic designs.
However, if the matrices have different 2-ranks, then obviously also the designs are non-isomorphic.
Thus we can conclude by Theorem~\ref{z} that there exist at least $1+6\lfloor\nu/3\rfloor$
non-isomorphic symmetric 2-$(2^{2\nu}-1,2^{2\nu-1},2^{2\nu-2})$ designs with distinguishing $2$-ranks.

\subsection*{Acknowledgments}
Research supported by {\em The Netherlands Organization for Scientific Research} $(NWO)$.
The authors would like to thank Ren\'e Peeters and Qing Xiang for helpful discussions.


\begin{thebibliography}{99}

\bibitem{AHPX2002}
K.T. Arasu, H.D.L. Hollmann, K. Player, Q. Xiang,  On the $p$-ranks of GMW difference sets,
\emph{Codes and Designs, Proceedings of a conference honoring Professor Dijen K. Ray-Chaudhuri
on the occasion of his 65th birthday} (2002), 9--35.

\bibitem{BE1992}
A.E. Brouwer, C.A. van Eijl, On the p-rank of the adjacency matrices of
strongly regular graphs, {\em Alg. Comb.} {\bf 1} (1992), 329-346.

\bibitem{BH}
A.E. Brouwer, W.H. Haemers, Spectra of graphs, Springer, 2012.

\bibitem{GM1982}
C.D. Godsil, B.D. McKay, Constructing cospectral graphs,
{\em Aequationes Math.} {\bf 25} (1982), 257--268.

\bibitem{HPR1999}
W.H. Haemers, M.J.P. Peeters, J.M. van Rijckevorsel, Binary codes of strongly regular graphs,
\emph{Des. Codes Cryptography} {\bf 17} (1999), 187--209.

\bibitem{MV2014}
A. Munemasa, F. Vanhove, Twisted symplectic polar graphs and Gorden-Mills-Welch differnce sets,
presentation at the Colloquium on Galois Geometry to the Memory of Fr\'{e}d\'{e}ric Vanhove, Ghent, February~2014.
\href{http://www.math.is.tohoku.ac.jp/\~munemasa/documents/20140228.pdf}{http://www.math.is.tohoku.ac.jp/\~munemasa/documents/20140228.pdf}

\bibitem{P1995}
M.J.P. Peeters, Uniqueness of strongly regular graphs having minimal p-rank,
{\em Linear Algebra Appl.} {\bf 226-228} (1995), 9--31.

\end{thebibliography}
\end{document}